\documentclass[a4paper,11pt,oneside]{article}
\usepackage{amsmath,amsthm,amsfonts,amssymb,ifpdf}
\usepackage{amssymb}
\usepackage{amsmath}
\usepackage{amsthm}
\usepackage{graphicx}
\usepackage[all]{xy}
\usepackage{enumerate}
\usepackage{tikz-cd}
\usepackage{subcaption}
\usepackage{authblk}

\ifpdf        
\usepackage[
pdftex,
colorlinks,%
linkcolor=blue,citecolor=red,urlcolor=blue,
hyperindex,%
plainpages=false,%
bookmarksopen,%
bookmarksnumbered%
]{hyperref} 
\usepackage{thumbpdf}
\else
\usepackage{hyperref}
\fi

\usepackage{theoremref} 

\usepackage{tikz}
\usetikzlibrary{positioning, arrows.meta}

\newtheorem*{claim*}{Claim}

\newtheorem{ex}{Example}[section]

\newtheorem{theorem}{Theorem}[section]
\newtheorem{lemma}{Lemma}[section]

\newtheorem{definition}{Definition}[section]

\numberwithin{equation}{section}
\providecommand{\keywords}[1]
{
	\small	
	\textbf{\textit{Keywords---}} #1
}
\providecommand{\msc}[1]
{
	\small	
	\textbf{\textit{Mathematics Subject Classification (2020) ---}} #1
}

\title{
A numerical scheme for a diffusion equation with nonlocal nonlinear boundary condition}

\author{Joydev Halder\thanks{halderjoydev@gmail.com}\ }
\author{Suman Kumar Tumuluri \thanks{suman.hcu@gmail.com}}

\affil{School of Mathematics and Statistics, University of Hyderabad, Hyderabad, India.}

\date{\today}
\begin{document}
\maketitle

\begin{abstract}
In this article, a numerical scheme to find approximate solutions to  the McKendrick-Von Foerster equation with diffusion (M-V-D) is presented. The main difficulty in employing the standard analysis to study the properties of this scheme is due to presence of nonlinear and nonlocal term in the Robin boundary condition in the M-V-D. To overcome this, we use the abstract theory of discretizations based on the notion of stability threshold to analyze the scheme. Stability, and convergence of the proposed numerical scheme are established. 
\end{abstract}
\\
\msc{35K20, 65M06, 65M12}
\section{Introduction}
The McKendrick-Von Foerster equation arises naturally in many areas of mathematical biology such as cell proliferation, and demography modeling (see \cite{Diekmann2000_book, Murray2002_book, Murray2003_book, Perthame2007, Perthame2015_book, Thieme2003_book}). In particular, the McKendrick–von Foerster equation is one amongst the important models whenever age structure is a vital feature in the modeling (see \cite{halder2020_1, Iannelli2017, inaba2017_book}). In the recent years, the McKendrick–Von Foerster equation with diffusion (M-V-D) has attracted interest of many engineers as well as mathematicians due to its applications in the modeling of  thermoelasticity, neuronal networks etc, (see \cite{Day1982_40, Day1985_book,  Michel2020_25,  Kakumani2016_39, Kakumani2017_22,Michel2013_36}). The main difficulty in the study of M-V-D is due to the nonlocal nature of the partial differential equation (PDE) and the boundary condition. The qualitative properties of the M-V-D equation have been developed by many authors. Though, numerical study of non-local equations got considerable focus, relatively less attention was paid to problems with the Robin boundary condition.\\
In this paper, our objective is to propose and analyze a numerical scheme to find approximate solutions to the following nonlinear diffusion equation
\begin{equation}\label{e1}
	\left\{
	\begin{aligned}
		&u_t(x,t)+u_x(x,t)+d\Big(x,s_1(t)\Big)u(x,t)=u_{xx}(x,t), \hskip 1 em x \in (0,a_\dagger), \hskip .5 em t> 0
		\\
		&u(0,t)-u_x(0,t)=\int_0^{a_\dagger}B\Big(x,s_2(t)\Big)u(x,t)dx, \hskip 1 em t \geq 0,
		\\
		&u({a_\dagger},t)=0, \ t\geq 0,
		\\
		&u(x,0)=u_0(x), \hskip 1 em x \in (0,a_\dagger),
		\\
		&s_i(t)=\int_0^{a_\dagger} \psi_i(x)u(x,t)dx, \hskip 1 em t \geq 0,
	\end{aligned}
	\right.
\end{equation}
where $a_\dagger>0$.	
In the given model, the unknown function $u(x,t)$ represents the age-specific density of individuals of age $x$ at time $t$. The function $d$ represents the death rate and depends on  $x$ and the environmental factor $s_1$. Similarly, the fertility rate $B$ depends on the age $x$ and the environmental factor $s_2$. Both the functions  $\psi_1$ and $\psi_2$ are  called the competition weight functions. Moreover, the functions $d$ and $B$ are assumed to be non-negative.
\\
In \cite{Kakumani2016_39}, the authors considered the  M-V-D with nonliner nonlocal Robin boundary condition and studied the existence and uniqueness of the solution. The authors of \cite{Kakumani2018_34} proposed a convergent numerical scheme to the M-V-D.  On the other hand, the existence of a global solution to the M-V-D in a bounded domain with nonliner nonlocal Robin boundary condition was proved when $d=d(x)$ in \cite{Kakumani2017_22}. Recently in \cite{halder_in}, an implicit finite difference scheme has been introduced to approximate the solution to the M-V-D  in a bounded domain with nonliner nonlocal Robin boundary condition at both the boundary points. Moreover, the wellposedness and the stability of the numerical scheme are proved using the method of upper and lower solution with the aid of the discrete maximum principle.
\\
The author of  \cite{marcos1991_22} presented an upwind scheme for a nonlinear hyperbolic integro-differential equation  with nonlocal boundary condition.  The analysis was carriedout employing  the general analytic framework developed in \cite{marcos1988_8, marcos1988_104, sanz1985_132}. The notion  of `stability with threshold' and a result due to Stetter (see \cite{stetter1973_23}, Lemma 1.2.1) were the most important tools for the analysis. 
\\
The above mentioned results inspired  us to propose an explicit finite difference numerical scheme to \eqref{e1}. The main difficulty in the analysis of the proposed numerical scheme is due to the nonlinearity and the Robin boundary condition that are presented in \eqref{e1}. The objective of this paper is to establish the stability and the convergence of our numerical method. Since the scheme is nonlinear, the standard techniques of proving stability (for instance, the Lax  theory etc). can not be used. Instead, the notion of nonlinear stability (with threshold)  is used to arrive at the convergence result.
\\
This article is organized as follows. In Section 2, we present a finite difference scheme and define the required  norms to use the general discretization framework. Moreover, we introduce the notion of  stability with  $h$-dependent thresholds. We prove consistency, stability and convergence results  in Section 3. Finaly, in Section 4, numerical examples are provided to justify the convergence results that are proved.
\section{The numerical scheme}
\noindent Let $h$, $k$ be the spacial and temporal step sizes. Denote by $(x_i,t^n)$ a typical grid point, where $x_i=ih$, and $t^n=nk$. Moreover, we fix $T>0$, assume that $a_\dagger = 2(M'+3)h$ for some $M'\in \mathbb{N}$ and $T=Nk$ for some $N \in \mathbb{N}$. To simplify the notations, we write $M=2(M'+3)$.
For every grid  point $(x_i,t^n)$, we denote the numerical solution by $U_i^n$, and set $\Psi_i=\psi(x_i)$, $\boldsymbol{U}^n=(U_1^n,U_2^n,U_3^n,\dots,U_{M-1}^n)$,  $\boldsymbol{\Psi}=(\Psi_1,\Psi_2,\Psi_3,\dots,\Psi_{M-1})$.\\
If $\boldsymbol{U}=(U_1,U_2,U_3,\dots,U_{M-1})$ and $\boldsymbol{V}=(V_1,V_2,V_3,\dots,V_{M-1})$, we define $$\mathcal{Q}_h(\boldsymbol{V})=\frac{4h}{3}(2V_1-V_2+2V_3) +\frac{h}{3}\sum\limits_{i=2}^{M'} (V_{2i}+4V_{2i+1}+V_{2i+2})+\frac{4h}{3}(2V_{2M'+3}-V_{2M'+4}+2V_{2M'+5}),$$ $$\boldsymbol{U\cdot V}=(U_1V_1,U_2V_2,U_3V_3,\dots,U_{M-1}V_{M-1}).$$
With the notation introduced so far, we propose the following scheme for \eqref{e1} using the forward difference approximation for $u_t$, the backward difference for $u_x$, and the centered in space discretization for $u_{xx}$: 
\begin{equation}\label{scheme}
	\left\{
	\begin{aligned}
		&\frac{U_i^{n+1}-U_i^n}{k}+\frac{U_{i}^n-U_{i-1}^n}{h}+d_i\bigg(\mathcal{Q}_h\Big(\boldsymbol{{\Psi}}_1 \cdot \boldsymbol{U}^n \Big)\bigg)U_i^n=\frac{U_{i+1}^n+U_{i-1}^n-2U_{i}^n}{h^2}, \quad 1 \leq i \leq M-1, 0 \leq n \leq N-1,
		\\
		& \left(1+\frac{1}{h}\right) U_0^n- \frac{1}{h}U_1^n=\mathcal{Q}_h\bigg(\boldsymbol{B}\Big(\mathcal{Q}_h(\boldsymbol{\Psi}_2 \cdot \boldsymbol{U}^n)\Big) \cdot \boldsymbol{U}^n\bigg), \quad 0 \leq n \leq N,
		\\
		&U_M^n=0,  \quad 0 \leq n \leq N,
		\\
		&U_i^0=u_0(x_i), 1 \leq i \leq M-1.
	\end{aligned}
	\right.
\end{equation}
In order to carry out the analysis within an abstract theory of discretizations, we introduce the general discretization framework. For, we define the spaces $$X_h = Y_h = \mathbb{R}^{N+1} \times (\mathbb{R}^{M-1})^{N+1}\times \mathbb{R}^{N+1}.$$
We also introduce the operator $\Phi_h \colon X_h \to Y_h$, defined through the formulae
$$ \Phi_h (\boldsymbol{V_0}, \boldsymbol{V^0}, \boldsymbol{V^1},...,\boldsymbol{V^N},\boldsymbol{V_M}) = (\boldsymbol{P_0}, \boldsymbol{P^0}, \boldsymbol{P^1},...,\boldsymbol{P^N},\boldsymbol{P_M}),$$
where
\begin{equation}
	\begin{aligned}
		&\boldsymbol{P}_0=(P_0^0,P_0^1, \cdots, P_0^N),
		\\
		&P_0^n=\left(1+\frac{1}{h}\right) V_0^n- \frac{1}{h}V_1^n-\mathcal{Q}_h\bigg(\boldsymbol{B}\Big(\mathcal{Q}_h(\boldsymbol{\Psi}_2 \cdot \boldsymbol{V}^n)\Big) \cdot \boldsymbol{V}^n\bigg), \quad 0\leq n \leq N,
		\\
		&\boldsymbol{P}_M=(P_M^0,P_M^1, \cdots, P_M^N),
		\\
		&P_M^n=\frac{V_M^n}{h}, \quad 0\leq n \leq N,
		\\
		&\boldsymbol{P}^n=(P_1^n,P_2^n, \dots,P_{M-1}^n), \quad 0\leq n \leq N,
		\\
		&P_i^0=V_i^0-U_i^0, \quad  1\leq i \leq M-1,
		\\
		&P_i^n= \frac{V_i^{n}-V_i^{n-1}}{k}+\frac{V_{i}^{n-1}-V_{i-1}^{n-1}}{h}+d_i\Big(\mathcal{Q}_h(\boldsymbol{{\Psi}}_1 \cdot \boldsymbol{V}^{n-1} )\Big)V_i^{n-1}
		\\
		&\quad \quad \quad-\frac{V_{i+1}^{n-1}+V_{i-1}^{n-1}-2V_{i}^{n-1}}{h^2}, \ 1\leq n \leq N, 1\leq i \leq M-1.
	\end{aligned}
\end{equation}
Now $\boldsymbol{U}_h = (\boldsymbol{U}_0, \boldsymbol{U}^0, \boldsymbol{U}^1,\cdots, \boldsymbol{U}^N) \in X_h$ is a solution to \eqref{scheme} if and only if it is a solution of the discrete problem 
\begin{equation}\label{equation}
	\Phi_h(\boldsymbol{U}_h)=\boldsymbol{0} \in Y_h.
\end{equation}
To investigate how close $\boldsymbol{U_h}$ is to $u$, we first need to  choose an element $\boldsymbol{u_h} \in X_h$, which is a suitable discrete representation of $u$. In particular, our choice is the set of nodal values of the theoretical solution $u$, namely
\begin{equation}\label{suitable}
	\begin{aligned}
		&\boldsymbol{u}_h = (\boldsymbol{u}_0, \boldsymbol{u}^0, \dots, \boldsymbol{u}^N,\boldsymbol{u}_M) \in X_h,
\end{aligned}
\end{equation}
where
\begin{equation}\label{adisc}
	\left\{
\begin{aligned}
		&\boldsymbol{u}_0=(u^0_0, u^1_0, \dots, u^N_0) \in \mathbb{R}^{N+1},\ u^n_0=u(0,t_n), \ 0\leq n \leq N,
		\\
		&\boldsymbol{u}^n=(u^n_1, u^n_2, \dots,u^n_{M-1}) \in \mathbb{R}^{M-1},\ u^n_i=u(x_i,t_n),\  1\leq i \leq M-1, \ 0\leq n \leq N,
		\\
		&\boldsymbol{u}_M=(u^0_M, u^1_M, \dots, u^N_M) \in \mathbb{R}^{N+1},\ u^n_M=u(a_\dagger,t_n), \ 0\leq n \leq N.
	\end{aligned}
\right.
\end{equation}
Then the global discretization error is defined to be the vector 
$$\boldsymbol{e_h=u_h- U_h} \in X_h,$$
and the local discretization error is given by 
$$\boldsymbol{I_h}= \Phi_h(\boldsymbol{u}_h) \in Y_h.$$
In order to measure the magnitude of errors, we define the following norms in the spaces $X_h$ and $Y_h$:
$$||(\boldsymbol{V}_0,\boldsymbol{V}^0,\boldsymbol{V}^1,...,\boldsymbol{V}^N,\boldsymbol{V}_M)||_{X_h}=h(||\boldsymbol{V}_0||_*+||\boldsymbol{V}_M||_*)+ \max \{||\boldsymbol{V}^0||, ||\boldsymbol{V}^1||,\dots, ||\boldsymbol{V}^N|| \} ,$$
$$||(\boldsymbol{P}_0,\boldsymbol{P}^0,\boldsymbol{P}^1,...,\boldsymbol{P}^N,\boldsymbol{P}_M)||_{Y_h}=\left(||\boldsymbol{P}_0||_*^2+||\boldsymbol{P}^0||^2+h||\boldsymbol{P}_M||_*^2+\sum\limits_{n=1}^N k||\boldsymbol{P}^n||^2\right)^{1/2},$$
where $||\boldsymbol{V}^n||^2=\sum\limits_{i=1}^{M-1} h|{V}_i^n|^2$ and $||\boldsymbol{V}_0||_*^2=\sum\limits_{n=0}^N k|{V}_0^n|^2$.\\
For $\boldsymbol{V} \in \mathbb{R}^{M-1}, \boldsymbol{Z} \in \mathbb{R}^{N+1}$, we define  $$\langle \boldsymbol{V}, \boldsymbol{W} \rangle= \sum\limits_{i=1}^M h V_i W_i,$$  $$||\boldsymbol{V}||_\infty= \max \limits_{1 \leq j \leq {M-1}} |V_i|, \quad ||\boldsymbol{Z}||^\infty= \max \limits_{0 \leq n \leq N} |Z^n|.$$ 
\\
For the sake of completeness, we give the following standard definitions (see \cite{marcos1991_22}).
\begin{definition}(Consistency)
	The discretization \eqref{equation} is said to be consistent with \eqref{e1} if 
	$$\lim \limits_{h \to 0} ||\Phi_h(\boldsymbol{u}_h)||_{Y_h}=\lim \limits_{h \to 0} ||\boldsymbol{I}_h||_{Y_h}=0.$$
	Moreover, if $ ||\boldsymbol{I}_h||_{Y_h}=\mathcal{O}(h^p)+\mathcal{O}(k^q)$ then we say that $(p,q)$ is the order of the consistency.
\end{definition}
\begin{definition}(Stability)
	The discretization \eqref{equation} is said to be stable restricted to the thresholds $R_h$ if there exists two positive constants $h_0$ and $S$ such that  for all  $h \leq h_0$ and for all $\boldsymbol{V}_h$, $\boldsymbol{W}_h$ in the open ball $B(\boldsymbol{u}_h,R_h)$ in $ X_h$ with center $\boldsymbol{u}_h$ and radius $R_h$ satisfying 
	$$||\boldsymbol{V}_h - \boldsymbol{W}_h||_{X_h} \leq S ||\Phi_h(\boldsymbol{V}_h)-\Phi_h(\boldsymbol{W}_h)||_{Y_h}.$$
\end{definition}
\begin{definition}(Convergence)
	The discretization \eqref{equation} is said to be convergent if there exists $h_0>0$ such that, for all  $h \leq h_0$, \eqref{equation} has a solution $\boldsymbol{U}_h$ for which 
	$$\lim \limits_{h \to 0} ||\boldsymbol{u}_h - \boldsymbol{U}_h||_{X_h}=\lim \limits_{h \to 0} ||\boldsymbol{e}_h||_{X_h}=0.$$
\end{definition}
The following Theorem, established in \cite{marcos1988_104} based on a result due to Stetter (see \cite{stetter1973_23}), is very important in our analysis.
\begin{theorem}(Cf. \cite{marcos1988_104})\label{marcos_them}
	Assume that \eqref{equation} is consistent and stable with thresholds $M_h$. If $\Phi_h$ is continuous in $B(u_h,M_h)$ and $||I_h||=\mathcal{O}(M_h)$ as $h \to 0$, then:\\
	$(i)$ for  sufficiently small $h>0$,  discrete equation \eqref{equation} admits a unique solution in $B(u_h, M_h)$.
	\\
	$(ii)$ the solutions to \eqref{equation} converge to the solution to \eqref{scheme} as $h$ tends to $0$. Furthermore, the order of convergence is not smaller than the order of consistency.
\end{theorem} 
\section{Consistency, stability and convergence}
In this section, we prove that numerical approximation \eqref{scheme} is consistent and stable. In order to prove the stability result, we first need to prove an elementary inequality. Finally, with the help of Theorem \ref{marcos_them}, we establish the convergence result. We begin with the consistency result in the following theorem.
\begin{theorem}[Consistency]\label{consistency}
	Assume that $d, B, \psi_i$, $i=1,2,$ are sufficiently smooth  such that the solution $u$ to \eqref{e1} is thrice continuously  differentiable  with bounded derivatives. Moreover, we assume  that there exists $L>0$ such that for every $0 \leq x \leq a_\dagger$, $s_1,s_2>0$,  $$|d(x,s_1)-d(x,s_2)| \leq L|s_1-s_2|,$$ and $$|B(x,s_1)-B(x,s_2)| \leq L|s_1-s_2|.$$ Then the  local discretization error satisfies
	$$||\Phi_h(u_h)||_{Y_h}=\{||\boldsymbol{U}^0-\boldsymbol{u}^0||^2+\mathcal{O}(h^2)+\mathcal{O}(k^2)\}^{1/2} , \ as\ h \to 0.$$
\end{theorem}
\begin{proof}Using the notation introduced in \eqref{adisc}, it is standard to verify that
	\begin{equation}\label{order_k}
		\sup\limits_{i,n} \left|\frac{u^{n}_i-u^{n-1}_i}{k} - u_t(x_i,t_{n-1})\right|=\mathcal{O}(k),\ \text{as}\ k \to 0,
	\end{equation} 
\begin{equation}\label{order_h}
		\sup\limits_{i,n}\left|\frac{u^{n-1}_i-u^{n-1}_{i-1}}{h} - u_x(x_i,t_{n-1})\right|=\mathcal{O}(h),\ \text{as}\ h \to 0,
\end{equation}
and
\begin{equation}\label{order_h2}
	\sup\limits_{i,n}\left|\frac{u^{n-1}_{i+1}+u^{n-1}_{i-1}-2u^{n-1}_{i}}{h^2} - u_{xx}(x_i,t_{n-1})\right|=\mathcal{O}(h^2),\ \text{as}\ h \to 0.
\end{equation}
	On the other hand, it is well known that if $f \in \mathcal{C}^1 [0,a_\dagger]$, then 
	\begin{equation}
		\left|\int_{0}^{a_\dagger} f(x)dx-\mathcal{Q}_h(\boldsymbol{f})\right|\leq Ch^4,
	\end{equation}
	where $C>0$ is independent of $h$.\\
	Lipschitz continuity of $d$ on compact sets readily implies
	\begin{equation}
		\begin{aligned}
             |d_i\Big(s_1(t_n)\Big)u^{n-1}_i-d_i\Big(\mathcal{Q}_h(\boldsymbol{\psi}_1 \cdot \boldsymbol{u}^{n-1})\Big)u^{n-1}_i|
              & \leq L |u^{n-1}_i| |s_1(t_n)- \mathcal{Q}_h(\boldsymbol{\psi}_1\boldsymbol{u}^{n-1})| 
              \\
              &\leq L C h^4 |u^{n-1}_i|.
		\end{aligned}
	\end{equation}
Hence we get 
	\begin{equation}\label{order_d}
	\begin{aligned}
		\sup \limits_{i,n}\left|d_i\Big(s_1(t_n)\Big)u^{n-1}_i-d_i\Big(\mathcal{Q}_h(\boldsymbol{\psi}_1 \cdot \boldsymbol{u}^{n-1})\Big)u^{n-1}_i\right|
		= \mathcal{O}(h^4),\ \text{as}\ h \to 0.
	\end{aligned}
\end{equation}
From the boundary condition, it follows that
\begin{equation}
	\begin{aligned}
		&\left| \int_0^{a_\dagger} B\Big(x,s_2(t_n)\Big)u(x,t_n)dx -\mathcal{Q}_h\bigg(\boldsymbol{B}\Big(\mathcal{Q}_h(\boldsymbol{\psi}_2 \cdot \boldsymbol{u}^n)\Big) \cdot \boldsymbol{u}^n\bigg)\right|
		\\
			&\leq \left| \int_0^{a_\dagger} B\Big(x,s_2(t_n)\Big)u(x,t_n)dx-\mathcal{Q}_h\bigg(\boldsymbol{B}\Big(s_2(t_n)\Big) \cdot \boldsymbol{u}^n\bigg) \right|
			\\
			&\quad+\left|\mathcal{Q}_h\bigg(\boldsymbol{B}\Big(s_2(t_n)\Big) \cdot \boldsymbol{u}^n\bigg)-\mathcal{Q}_h\bigg(\boldsymbol{B}\Big(\mathcal{Q}_h(\boldsymbol{\psi}_2 \cdot \boldsymbol{u}^n)\Big) \cdot \boldsymbol{u}^n\bigg)\right|
			\\
			& \leq Ch^4+ |\mathcal{Q}_h\Big(\left|s_2(t_n)-\mathcal{Q}_h(\boldsymbol{\Psi}_2\cdot \boldsymbol{u}^n)\right|L\boldsymbol{u}^n\Big)|
			\\
			& \leq Ch^4+LCh^4 a_\dagger||\boldsymbol{u}^n||_\infty.
	\end{aligned}
\end{equation}
Therefore we can write
\begin{equation}\label{order_int}
	\begin{aligned}
		\sup \limits_{n} \left| \int_0^{a_\dagger} B\Big(x,s_2(t_n)\Big)u(x,t_n)dx -\mathcal{Q}_h\bigg(\boldsymbol{B}\Big(\mathcal{Q}_h(\boldsymbol{\psi}_2 \cdot \boldsymbol{u}^n)\Big) \cdot \boldsymbol{u}^n\bigg)\right|
		= \mathcal{O}(h^4),\ \text{as}\ h \to 0.
	\end{aligned}
\end{equation}
Using \eqref{order_k}, \eqref{order_h}, \eqref{order_h2}, \eqref{order_d} and \eqref{order_int}, one can easily conclude the proof of the required result. 
\end{proof}
To prove  numerical scheme \eqref{scheme} is stable, we need the following lemma.
\begin{lemma}\label{l1}
	If $x$, $y$, $a$, $b$ and $h$ be five positive real numbers such that $\left(1+\frac{1}{h}\right)x-\frac{1}{h}y\leq a+b,$  then  $\left(1+\frac{1}{h}\right)x^2-\frac{1}{h}y^2\leq 2(a^2+b^2)$. 
\end{lemma}
\begin{proof}
	Consider
	\begin{equation*}
		\begin{aligned}
			\left(1+\frac{1}{h}\right)x^2-\frac{1}{h}y^2&\leq \left(1+\frac{1}{h}\right)\left(\frac{ah+bh+y}{h+1}\right)^2-\frac{1}{h}y^2\\
			&=\frac{a^2h^2+b^2h^2+2abh^2+y^2+2yh(a+b)}{h(h+1)}-\frac{1}{h}y^2\\
			&\leq \frac{2h(h+1)(a^2+b^2)+(h+1)y^2}{h(h+1)}-\frac{1}{h}y^2\\
			&=2(a^2+b^2).
		\end{aligned}
	\end{equation*}
	This completes the proof of the Lemma.
\end{proof}
\noindent Now, we are ready to establish the following stability theorem.
\begin{theorem}[Stability]\label{stability}
	Assume the hypotheses of Theorem \ref{consistency}. Let $r$ and $\lambda$ be such that  $k=rh^2=\lambda h$ with $\lambda+2r\leq1$. Then the discretization \eqref{equation} is stable with thresholds $R_h=Rh$, where $R$ is a fixed positive constant. 
\end{theorem} 
\begin{proof}
	Assume that $\boldsymbol{u}_h \in X_h$ is the discrete representation of $u$ given in \eqref{suitable}--\eqref{adisc}.	Suppose $\boldsymbol{V}_h$, $\boldsymbol{W}_h$ belong to the ball $B(\boldsymbol{u}_h, L_h)$. We set
	$$\boldsymbol{V}_h = (\boldsymbol{V_0}, \boldsymbol{V^0}, \boldsymbol{V^1},...,\boldsymbol{V^N},\boldsymbol{V_M}), \quad \Phi(\boldsymbol{V}_h) = (\boldsymbol{P_0}, \boldsymbol{P^0}, \boldsymbol{P^1},...,\boldsymbol{P^N},\boldsymbol{P_M}),$$
	$$\boldsymbol{W}_h = (\boldsymbol{W_0}, \boldsymbol{W^0}, \boldsymbol{W^1},...,\boldsymbol{W^N},\boldsymbol{W_M}), \quad \Phi(\boldsymbol{W}_h) = (\boldsymbol{R_0}, \boldsymbol{R^0}, \boldsymbol{R^1},...,\boldsymbol{R^N},\boldsymbol{R_M}).$$
	Then from the definition of the norm  in  $X_h$, we find that
	\begin{equation*}
		\begin{aligned}
			Rh &\geq ||\boldsymbol{V}_h- \boldsymbol{u}_h||_{X_h}
			\\
			&= h(||\boldsymbol{V}_0- \boldsymbol{u}_0||_*+||\boldsymbol{V}_M- \boldsymbol{u}_M||_*)+\max\limits_{0 \leq n \leq N} \{||\boldsymbol{V}^n-\boldsymbol{u}^n||\} 
			\\
			& \geq \left( \sum\limits_{i=1}^{M-1} h|{V}_i^n-{u}_i^n|^2 \right)^{\frac{1}{2}},
		\end{aligned}
	\end{equation*}
or 
\begin{equation*}
	R\sqrt{h}  \geq |{V}_i^n-{u}_i^n|, \ 0\leq n \leq N,\ 1\leq i \leq M-1.
\end{equation*}
This readily implies
\begin{equation*}
	R\sqrt{h} \geq  ||\boldsymbol{V}^n- \boldsymbol{u}^n ||_\infty, \ 0\leq n \leq N.
\end{equation*}
Therefore there exists $C>0$, independent of $n$, such that 
	\begin{equation}\label{w_bound}
	\begin{aligned}
		|| \boldsymbol{V}^n ||_\infty & \leq R\sqrt{h}  +|| \boldsymbol{u}^n ||_\infty 
		\leq C, \ 0\leq n \leq N.
	\end{aligned}
\end{equation}
On the other hand, from the definition of $\Phi_h$, we obtain
	\begin{equation}
		\begin{aligned}
			V_i^n- W_i^n=&(1-\lambda-2r)(V_i^{n-1}- W_i^{n-1})+(r+\lambda)(V_{i-1}^{n-1}- W_{i-1}^{n-1})+r(V_{i+1}^{n-1}- W_{i+1}^{n-1})
			\\
			&+k(P_i^n-R_i^n)-k\big[d_i\Big(Q_h(\boldsymbol{{\Psi}}_1 \cdot \boldsymbol{V}^{n-1} )\Big)V_i^{n-1}-d_i\Big(Q_h(\boldsymbol{{\Psi}}_1\cdot \boldsymbol{W}^{n-1} )\Big)W_i^{n-1}\big],
		\end{aligned}
	\end{equation}
	for $1\leq n\leq N$, $1 \leq i \leq M-1$. Multiplication by $h(V_i^n-W_i^n)$, summation over $1\leq i \leq M-1$, and  the Cauchy-Schwarz inequality yield
	\begin{equation}\label{e10}
		\begin{aligned}
			||\boldsymbol{V}^n- \boldsymbol{W}^n||^2\leq &\frac{1}{2}||\boldsymbol{V}^{n-1}- \boldsymbol{W}^{n-1}||^2+(\frac{1}{2}+k)||\boldsymbol{V}^n-\boldsymbol{ W}^n||^2
			\\
			&+\frac{k}{2}||\boldsymbol{d}\Big(Q_h(\boldsymbol{{\Psi}}_1 \cdot \boldsymbol{V}^{n-1} )\Big)\boldsymbol{V}^{n-1}-\boldsymbol{d}\Big(Q_h(\boldsymbol{{\Psi}}_1 \cdot \boldsymbol{W}^{n-1} )\Big)\boldsymbol{W}^{n-1}||^2+\frac{k}{2}||\boldsymbol{P}^n-\boldsymbol{R}^n||^2
			\\
			&+\frac{k}{2}\left((1+\frac{1}{h})|V_0^{n-1}-W_0^{n-1}|^2-\frac{1}{h}|V_1^{n-1}-W_1^{n-1}|^2+\frac{1}{h}|V_M^{n-1}-W_M^{n-1}|^2\right),
		\end{aligned}
	\end{equation}
	where we have used the fact that $2r+\lambda\leq1$.
	Now consider
		\begin{align}
			&||\boldsymbol{d}\Big(Q_h(\boldsymbol{{\Psi}}_1 \cdot \boldsymbol{V}^{n-1} )\Big)\boldsymbol{V}^{n-1}-\boldsymbol{d}\Big(Q_h(\boldsymbol{{\Psi}}_1 \cdot \boldsymbol{W}^{n-1} )\Big)\boldsymbol{W}^{n-1}|| \nonumber
			\\
			 &\leq ||\boldsymbol{d}\Big(Q_h(\boldsymbol{{\Psi}}_1 \cdot \boldsymbol{V}^{n-1})\Big)||_{\infty}||\boldsymbol{V}^{n-1}- \boldsymbol{W}^{n-1}||+||\boldsymbol{W}^{n-1}||_{\infty}||\boldsymbol{d}\Big(Q_h(\boldsymbol{{\Psi}}_1 \cdot \boldsymbol{V}^{n-1})\Big)- \boldsymbol{d}\Big(Q_h(\boldsymbol{{\Psi}}_1 \cdot \boldsymbol{W}^{n-1})\Big)|| \nonumber
			 \\
			&\leq C||\boldsymbol{V}^{n-1}- \boldsymbol{W}^{n-1}||, \label{d_bound}
		\end{align}
for some $C>0$ independent of $h,k$. Thus  \eqref{e10}--\eqref{d_bound} together give
	\begin{equation}\label{e11}
		\begin{aligned}
			(1-2k) ||\boldsymbol{V}^n- \boldsymbol{W}^n||^2\leq &(1+Ck)||\boldsymbol{V}^{n-1}- \boldsymbol{W}^{n-1}||^2+{k}||\boldsymbol{P}^n-\boldsymbol{R}^n||^2\\&
			+{k}\left((1+\frac{1}{h})|V_0^{n-1}-W_0^{n-1}|^2-\frac{1}{h}|V_1^{n-1}-W_1^{n-1}|^2+\frac{1}{h}|V_M^{n-1}-W_M^{n-1}|^2\right).
		\end{aligned}
	\end{equation}
	Using \eqref{w_bound} and the boundary condition, we can write
		\begin{align}
			(1+\frac{1}{h})|V_0^{n}-W_0^{n}|-\frac{1}{h}|V_1^{n}-W_1^{n}|  \nonumber
			\leq& |P_0^{n}-R_0^{n}|
			\\
			&+ |\mathcal{Q}_h\bigg(\boldsymbol{B}\Big(\mathcal{Q}_h(\boldsymbol{\Psi}_2 \cdot \boldsymbol{V}^n)\Big) \cdot \boldsymbol{V}^n\bigg)
			- \mathcal{Q}_h\bigg(\boldsymbol{B}\Big(\mathcal{Q}_h(\boldsymbol{\Psi}_2 \cdot \boldsymbol{W}^n)\Big) \cdot \boldsymbol{W}^n\bigg)| \nonumber
			\\
			\leq& |P_0^{n}-R_0^{n}|
			+ |\mathcal{Q}_h\bigg(\boldsymbol{B}\Big(\mathcal{Q}_h(\boldsymbol{\Psi}_2 \cdot \boldsymbol{V}^n)\Big) \cdot ( \boldsymbol{V}^n-\boldsymbol{W}^n)\bigg)| \nonumber
			\\
			&\quad \quad +| \mathcal{Q}_h\bigg(\Big(\boldsymbol{B}\big(\mathcal{Q}_h(\boldsymbol{\Psi}_2 \cdot \boldsymbol{V}^n)\big)-\boldsymbol{B}\big(\mathcal{Q}_h(\boldsymbol{\Psi}_2 \cdot \boldsymbol{W}^n)\big)\Big) \cdot \boldsymbol{W}^n\bigg)| \nonumber
			\\
			\leq &|P_0^{n}-R_0^{n}|
			+ ||\boldsymbol{B}||_\infty |\mathcal{Q}_h( \boldsymbol{V}^n-\boldsymbol{W}^n)| \nonumber
			\\
			&+L a_\dagger ||\boldsymbol{\Psi}_2||_\infty |\mathcal{Q}_h( \boldsymbol{V}^n-\boldsymbol{W}^n)|\ ||\boldsymbol{W}^n||_\infty \nonumber
			\\
			\leq & |P_0^{n}-R_0^{n}|
			+ C||\boldsymbol{V}^n- \boldsymbol{W}^n||, \label{bounday}
		\end{align}
	for some $C>0$ independent of mesh sizes $h$ and $k$. From \eqref{bounday} and Lemma \ref{l1}, we deduce that
	\begin{equation}
		\begin{aligned}
			(1+\frac{1}{h})|V_0^{n-1}-W_0^{n-1}|^2-\frac{1}{h}|V_1^{n-1}-W_1^{n-1}|^2 \leq C \left(||\boldsymbol{V}^{n-1}- \boldsymbol{W}^{n-1}||^2+|P_0^{n-1}-R_0^{n-1}|^2\right).
		\end{aligned}
	\end{equation}
	On substituting  this bound in \eqref{e11}, we obtain
	\begin{align}
		||\boldsymbol{V}^n- \boldsymbol{W}^n||^2\leq \frac{1+Ck}{1-2k}&||\boldsymbol{V}^{n-1}- \boldsymbol{W}^{n-1}||^2+\frac{Ck}{1-2k}\left(||\boldsymbol{P}^n-\boldsymbol{R}^n||^2\right. \nonumber
		\\
		&\left.+|P_0^{n-1}-R_0^{n-1}|^2+h|P_M^{n-1}-R_M^{n-1}|^2\right). \label{e12}
	\end{align}
	From the discrete Gronwall lemma, there exists $C_T$ depending solely on $T$ such that 
		\begin{align}
			||\boldsymbol{V}^n- \boldsymbol{W}^n||^2\leq C_T &\bigg\{||\boldsymbol{V}^0- \boldsymbol{W}^0||^2+\frac{Ck}{1-2k}\sum\limits_{m=1}^n\Big( ||\boldsymbol{P}^m-\boldsymbol{R}^m||^2 \Big.\bigg. \nonumber
			\\
			&\bigg. \Big.+|P_0^{m-1}-R_0^{m-1}|^2+h|P_M^{m-1}-R_M^{m-1}|^2\Big)\bigg\}. \label{gron}
		\end{align}
	Thus  for $k$ sufficiently small, this immediately gives
		\begin{align}
			||\boldsymbol{V}^n- \boldsymbol{W}^n||\leq C_T &\bigg\{||\boldsymbol{P}^0- \boldsymbol{R}^0||^2+C\left(\sum\limits_{m=1}^N k ||\boldsymbol{P}^m-\boldsymbol{R}^m||^2\right)\bigg. \nonumber
			\\
			&\bigg.+C||\boldsymbol{P}_0-\boldsymbol{R}_0||_*^2+h||\boldsymbol{P}_M-\boldsymbol{R}_M||_*^2\bigg\}^{\frac{1}{2}}. \label{part1}
		\end{align}
	Again, from \eqref{bounday}, we have
	\begin{equation*}
		\begin{aligned}
			(1+h)|V_0^{n}-W_0^{n}|-|V_1^{n}-W_1^{n} |\leq h\left( C||\boldsymbol{V}^{n}- \boldsymbol{W}^{n}||+|P_0^{n}-R_0^{n}|\right).
		\end{aligned}
	\end{equation*}
	Multiplying both sides with $|V_0^n- W_0^n|$ and using the AM-GM inequality,  we get
	\begin{equation}\label{d1}
		\begin{aligned}
			|V_0^{n}-W_0^{n}|^2 \leq |V_1^{n}-W_1^{n} |^2+ h\left( C||\boldsymbol{V}^{n}- \boldsymbol{W}^{n}||^2+|P_0^{n}-R_0^{n}|^2\right).
		\end{aligned}
	\end{equation}
	Multiplying both sides by $hk$, taking summation on $n$ yields	
		\begin{align}
			h||\boldsymbol{V}_0-\boldsymbol{W}_0||^2_* 
			&\leq \sum\limits_{n=0}^N hk |V_1^n -W_1^n |^2+  \sum\limits_{n=0}^N k h^2\left( C||\boldsymbol{V}^{n}- \boldsymbol{W}^{n}||^2+|P_0^{n}-R_0^{n}|^2\right) \nonumber
			\\
			& \leq (1+Ch^2)  \sum\limits_{n=0}^N k ||\boldsymbol{V}^n -\boldsymbol{W}^n ||^2+  h^2||\boldsymbol{P}_{0}- \boldsymbol{R}_{0}||_*^2. \label{part2}
		\end{align}
The second boundary condition immediately gives
\begin{equation}\label{part3}
	||\boldsymbol{V}_M-\boldsymbol{W}_M||_* = h ||\boldsymbol{P}_M-\boldsymbol{R}_M||_*.
\end{equation}
	From \eqref{part1}, \eqref{part2} and \eqref{part3}, we observe that
	\begin{equation*}\label{final}
		\begin{aligned}
			h\Big(||\boldsymbol{V}_0- \boldsymbol{W}_0||_*&+||\boldsymbol{V}_M-\boldsymbol{W}_M||_* \Big)+\max \Big\{||\boldsymbol{V}^0-\boldsymbol{W}^0||, ||\boldsymbol{V}^1-\boldsymbol{W}^1||,\dots, ||\boldsymbol{V}^N-\boldsymbol{W}^N||\Big\}  
			\\
			&\leq K\left( ||\boldsymbol{P}_0-\boldsymbol{R}_0||_*^2+||\boldsymbol{P}^0-\boldsymbol{R}^0||^2+h||\boldsymbol{P}_M-\boldsymbol{R}_M||_*^2+\sum\limits_{m=1}^N k ||\boldsymbol{P}^m-\boldsymbol{R}^m||^2\right)^{\frac{1}{2}},
		\end{aligned}
	\end{equation*}
	where $K$ is a constant.
	This completes the proof.
\end{proof}
In the following result, we establish that \eqref{scheme} is indeed a convergent scheme.
\begin{theorem}[Convergence]\label{convergence}
	Assume the hypotheses of Theorem \ref{consistency}. If $$ ||\boldsymbol{U}^0- \boldsymbol{u}^0||_{X_h}=\mathcal{O}(h),\ \text{as}\ h \to 0,$$ then the discretization \eqref{equation} is convergent. 
\end{theorem}
\begin{proof}
	The proof is an immediate\textbf{} consequence of Theorem \ref{marcos_them}, Theorem \ref{consistency}, and  Theorem \ref{stability}.
\end{proof}
\subsection{Non-homogeneous  boundary condition at $x=a_\dagger$} \label{subsec:non_homo_bc}
In this subsection, we consider \eqref{e1} with non-homogeneous Dirichlet boundary condition i.e.,
\begin{equation}\label{non_e1}
	\left\{
	\begin{aligned}
		&u_t(x,t)+u_x(x,t)+d\big(x,s_1(t)\big)u(x,t)=u_{xx}(x,t), \hskip 1 em x \in (0,a_\dagger), \hskip .5 em t> 0,
		\\
		&u(0,t)-u_x(0,t)=\int_0^{a_\dagger}B\Big(x,s_2(t)\Big)u(x,t)dx, \hskip 1 em t \geq 0,
		\\
		&u({a_\dagger},t)=g(t), \ t\geq 0,
		\\
		&u(x,0)=u_0(x), \hskip 1 em x \in (0,a_\dagger),
		\\
		&s_i(t)=\int_0^{a_\dagger} \psi_i(x)u(x,t)dx, \hskip 1 em t \geq 0,
	\end{aligned}
	\right.
\end{equation}
where $g$ is a given smooth function. 
In order to discretize \eqref{non_e1}, we use the notations from the previous section. Let $h, \ k, \ T, \ M$ be as Section 2 and $U_i^n$ denote the approximate solution to \eqref{non_e1} at the grid point $(x_i,t^n)$. Moreover, we define $g^n= g(t^n)$, $0 \leq n \leq N$.

By discretizing \eqref{non_e1} as in Section 2, we arrive at the following finite difference scheme (see \eqref{scheme}) 
\begin{equation}\label{new_scheme}
	\left\{
	\begin{aligned}
		&\frac{U_i^{n+1}-U_i^n}{k}+\frac{U_{i}^n-U_{i-1}^n}{h}+d_i\bigg(\mathcal{Q}_h\Big(\boldsymbol{{\Psi}}_1 \cdot \boldsymbol{U}^n \Big)\bigg)U_i^n=\frac{U_{i+1}^n+U_{i-1}^n-2U_{i}^n}{h^2}, \quad 1 \leq i \leq M-1,\ 0 \leq n \leq N-1,
		\\
		& \left(1+\frac{1}{h}\right) U_0^n- \frac{1}{h}U_1^n=\mathcal{Q}_h\bigg(\boldsymbol{B}\Big(\mathcal{Q}_h(\boldsymbol{\Psi}_2 \cdot \boldsymbol{U}^n)\Big) \cdot \boldsymbol{U}^n\bigg), \quad 0 \leq n \leq N,
		\\
		&U_M^n=g^n,  \quad 0 \leq n \leq N,
		\\
		&U_i^0=u_0(x_i), 1 \leq i \leq M-1.
	\end{aligned}
	\right.
\end{equation}
As before, to carryout the analysis, we use the spaces $X_h$ and $ Y_h $ that are introduced in Section 2. Moreover,
we  consider the operator $ \widetilde{\Phi}_h  \colon X_h \to Y_h$, defined through the formulae
$$ \widetilde{\Phi}_h (\boldsymbol{V_0}, \boldsymbol{V^0}, \boldsymbol{V^1},...,\boldsymbol{V^N},\boldsymbol{V_M}) = (\boldsymbol{P_0}, \boldsymbol{P^0}, \boldsymbol{P^1},...,\boldsymbol{P^N},\boldsymbol{P_M}),$$
where
\begin{equation}
	\begin{aligned}
		&\boldsymbol{P}_0=(P_0^0,P_0^1, \cdots, P_0^N),
		\\
		&P_0^n=\left(1+\frac{1}{h}\right) V_0^n- \frac{1}{h}V_1^n-\mathcal{Q}_h\bigg(\boldsymbol{B}\Big(\mathcal{Q}_h(\boldsymbol{\Psi}_2 \cdot \boldsymbol{V}^n)\Big) \cdot \boldsymbol{V}^n\bigg), \quad 0\leq n \leq N,
		\\
		&\boldsymbol{P}_M=(P_M^0,P_M^1, \cdots, P_M^N),
		\\
		&P_M^n=\frac{V_M^n-g^n}{h}, \quad 0\leq n \leq N,
		\\
		&\boldsymbol{P}^n=(P_1^n,P_2^n, \dots,P_{M-1}^n), \quad 0\leq n \leq N,
		\\
		&P_i^0=V_i^0-U_i^0, \quad  1\leq i \leq M-1,
		\\
		&P_i^n= \frac{V_i^{n}-V_i^{n-1}}{k}+\frac{V_{i}^{n-1}-V_{i-1}^{n-1}}{h}+d_i\Big(\mathcal{Q}_h(\boldsymbol{{\Psi}}_1 \cdot \boldsymbol{V}^{n-1} )\Big)V_i^{n-1}
		\\
		&\quad \quad \quad-\frac{V_{i+1}^{n-1}+V_{i-1}^{n-1}-2V_{i}^{n-1}}{h^2}, \ 1\leq n \leq N, 1\leq i \leq M-1.
	\end{aligned}
\end{equation}
Using the definition of  $ \widetilde{\Phi}_h $, and the arguments used in Theorem \ref{consistency}, Theorem \ref{stability} and Theorem \ref{convergence}, one can easily show that \eqref{new_scheme} is indeed a convergent scheme whenever the hypotheses in Theorem \ref{convergence} hold.
\section{Numerical simulations}
In this section, we present some examples in which the numerical solutions to \eqref{e1} are computed using  \eqref{scheme}  to validate the results in the previous section. All the computations have been performed using  Matlab 8.5. In all the examples, we have taken $a_\dagger =1$, $\psi_1(x) \equiv \psi_2 (x) \equiv 1$ and $r=0.4$.  
	\begin{ex}\label{example1}
\end{ex}
\noindent In order to test the numerical scheme, assume that $u_0$, $d$,  and $B$  are given by\medskip \\ \centerline{$u_0(x)=e-e^x$, $d(x,s)=1$,\ $B(x,s)=e$, \ $x \in (0,1),\ s\geq 0$.}\medskip \\
Note that, for the given set of functions, $u(x)= (e-e^x)e^{-t}$ is a solution to \eqref{e1}. One can easily check that $d$ and $B$ satisfy hypotheses of Theorem \ref{convergence}. Hence \eqref{scheme} is a convergent  numerical scheme. \\
\begin{figure}[h!]
	\centering
	\begin{subfigure}[b]{0.49\linewidth}
		\includegraphics[width=\linewidth]{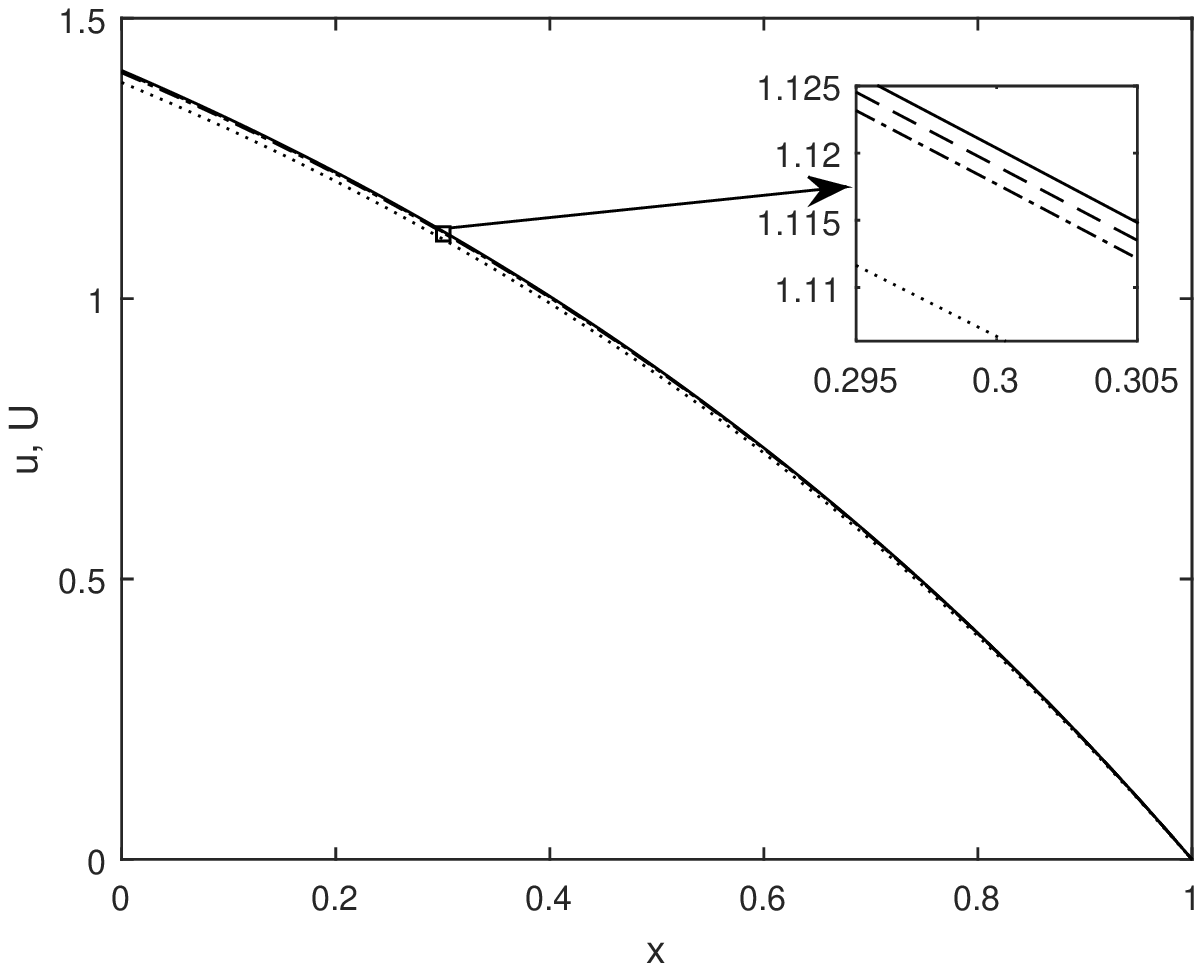}
		\label{fig:steady1}
	\end{subfigure}
	\begin{subfigure}[b]{0.49\linewidth}
		\includegraphics[width=\linewidth]{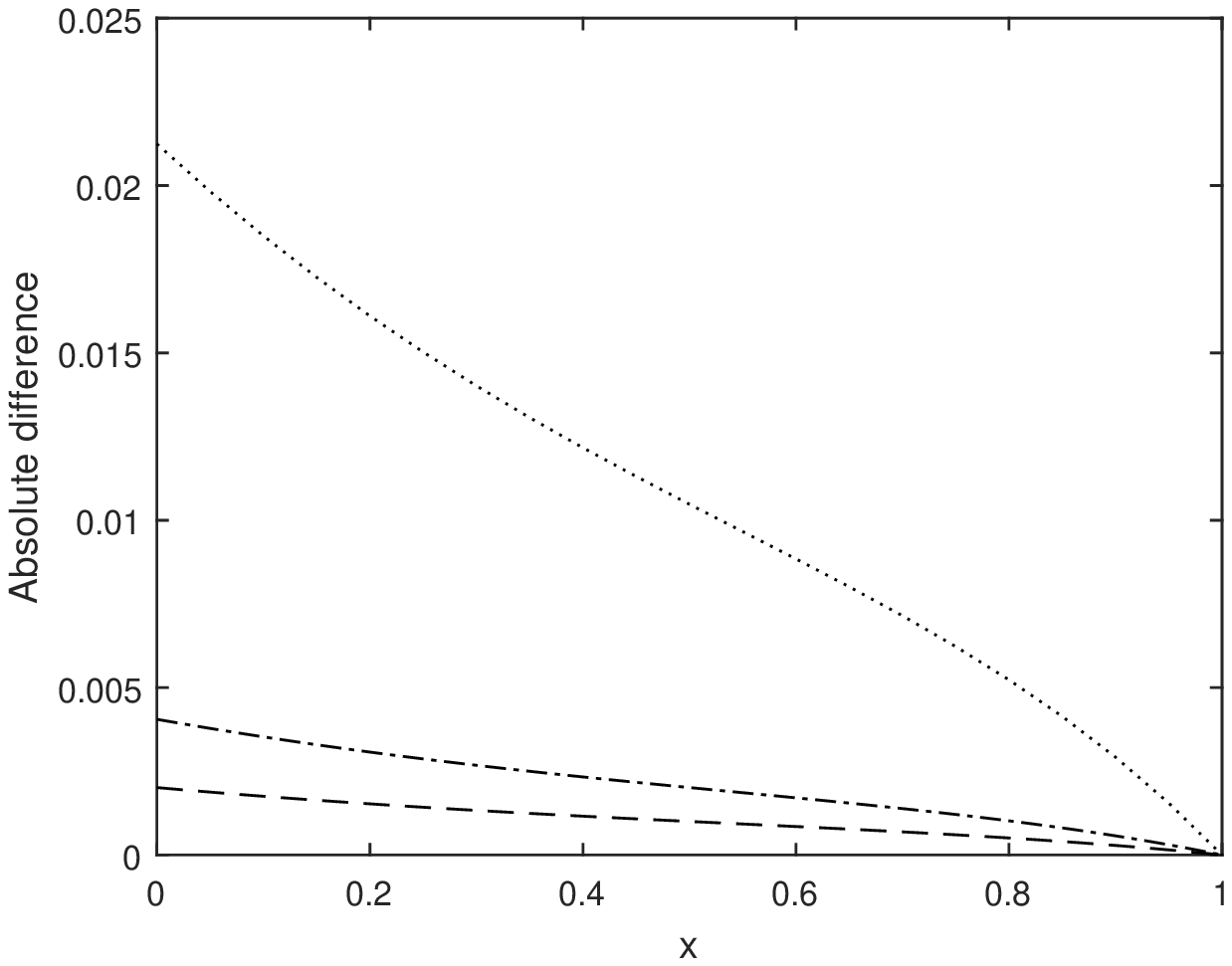}
		\label{fig:error1}
	\end{subfigure}
	\caption{The exact solution to \eqref{e1}, and the approximate solutions using \eqref{scheme}  with $d(x,s)$, $B(x,s)$ given in Example \ref{example1}; Left: $u(x,0.2)$ (solid line), $\boldsymbol{U}_{0.05}$ (dotted line), $\boldsymbol{U}_{0.01}$ (dash-dotted line), $\boldsymbol{U}_{0.005}$ (dashed line) for $0\leq x \leq 1$, Right: $\mid u(x,0.2)- \boldsymbol{U}_{0.05}\mid$ (dotted line), $\mid u(x,0.2)- \boldsymbol{U}_{0.01}\mid$ (dash-dotted line) and $\mid u(x,0.2)- \boldsymbol{U}_{0.005}\mid$ (dashed line).}
	\label{1fig1}
\end{figure}
In Figure \ref{1fig1}, we show the absolute difference between the exact solution and the computed solutions. 
In Figure \ref{1fig1} (left), we present the exact solution $u$ to \eqref{e1} and the corresponding numerical solutions to \eqref{scheme} with $h=0.05,\ 0.01, \ 0.005$ at $t=0.2$. From this figure, it is evident that $\boldsymbol{U}_{0.05}$, $\boldsymbol{U}_{0.01}$ and $\boldsymbol{U}_{0.005}$ are very close to $u$  at $t=0.2$. This phenomenon re-validates the results that are proved in Theorem \ref{convergence}.  In Figure \ref{1fig1} (right), the  difference between $u(x,0.2)$ and $\boldsymbol{U}_{h}$ at $t=0.2$, with $h=0.05,\ 0.01, \ 0.005$ are shown. From these figures, we can conclude that the sequence $\boldsymbol{U}_h$ indeed converges to the solution $u$ at $t=0.2$ as mentioned in Theorem \ref{convergence}.
	\begin{ex}\label{example2}
\end{ex}
\noindent In this example, we assume that $u_0$, $d$,  and $B$  are given by\medskip \\ \centerline{$u_0(x)=e-e^x$, $d(x,s)=\frac{1}{2}+\frac{s}{1-e^{-1}}$,\ $B(x,s)=2e^x$, \ $x \in (0,1),\ s\geq 0$.}\medskip \\
It is easy to verify that $d$ and $B$ satisfy hypotheses of Theorem \ref{convergence} . Therefore \eqref{scheme} is a convergent  numerical scheme. 
\begin{figure}[h!]
	\centering
	\begin{subfigure}[b]{0.49\linewidth}
		\includegraphics[width=\linewidth]{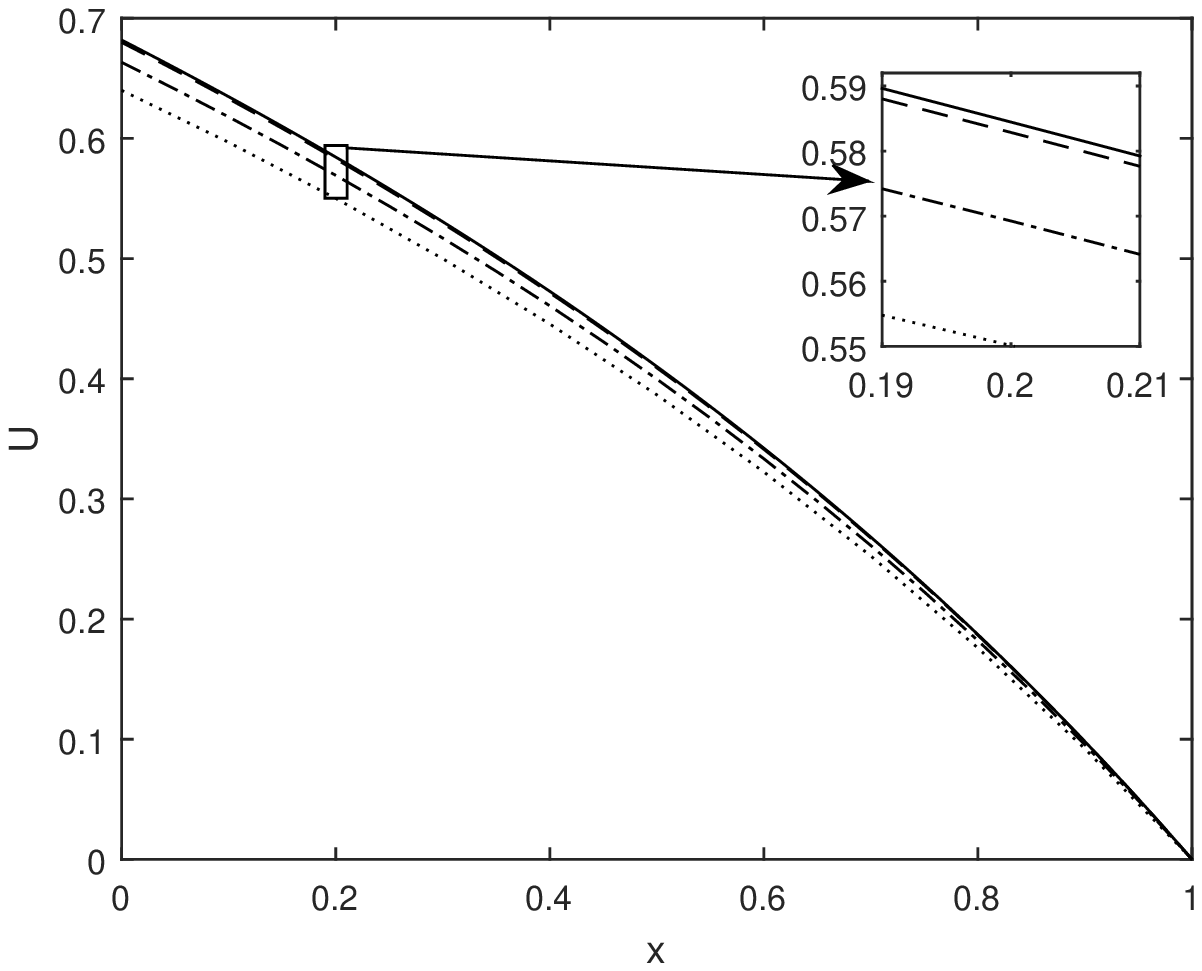}
		\label{fig:steady1}
	\end{subfigure}
	\begin{subfigure}[b]{0.49\linewidth}
		\includegraphics[width=\linewidth]{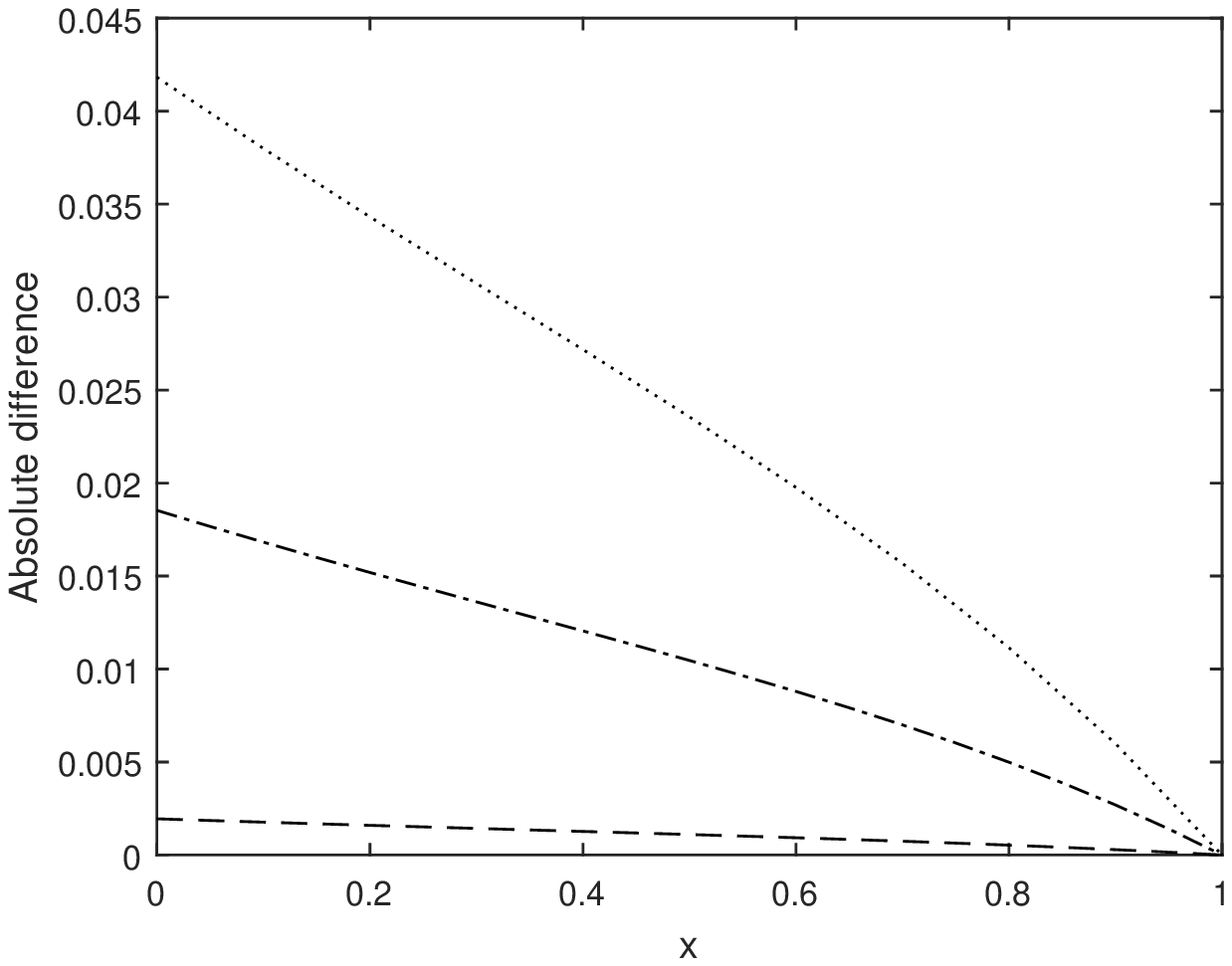}
		\label{fig:error1}
	\end{subfigure}
	\caption{The approximate solutions to \eqref{scheme}  with $d(x,s)$, $B(x,s)$ given in Example \ref{example2}; Left: $\boldsymbol{U}_{0.005}$ (solid line), $\boldsymbol{U}_{0.01}$ (dashed line),  $\boldsymbol{U}_{0.05}$ (dash-dotted line), $\boldsymbol{U}_{0.1}$ (dotted line) for $0\leq x \leq 1$, Right: $\mid \boldsymbol{U}_{0.005}- \boldsymbol{U}_{0.01}\mid$ (dashed line), $\mid \boldsymbol{U}_{0.005}- \boldsymbol{U}_{0.05}\mid$ (dash-dotted line), and $\mid \boldsymbol{U}_{0.005}- \boldsymbol{U}_{0.1}\mid$ (dotted line).}
	\label{1fig3}
\end{figure}
In Figure \ref{1fig3}, we plot the computed solutions for different values of $h$ at $t=0.8$. In Figure \ref{1fig3} (left), solution to \eqref{scheme} with $h=0.1,\ 0.05,\ 0.01,\ 0.005$ at $t=0.8$ are presented. From these figures, it is straightforward  to see that $\boldsymbol{U}_{0.005}$ is closer to $\boldsymbol{U}_{0.05}$ than $\boldsymbol{U}_{0.1}$ and $\boldsymbol{U}_{0.005}$ is closer to $\boldsymbol{U}_{0.01}$ than $\boldsymbol{U}_{0.05}$at $t=0.8$.  In Figure \ref{1fig3} (right), we plot absolute difference between $\boldsymbol{U}_{0.005}$ and $\boldsymbol{U}_{h}$ with $h=0.1,\  0.05,\ 0.01$  at $t=0.8$. From these graphs, one can observe that  the numerical solutions to \eqref{scheme} are converging.
	\begin{ex}\label{example3}
\end{ex}
\noindent In this example, we consider the  non-homogeneous case described in Subsection \ref{subsec:non_homo_bc}. In particular, we consider \eqref{non_e1} with $u_0$, $d,\ B$,  and $g$  are given by\medskip \\ \centerline{$u_0=\frac{e^{-x}}{2}$, $d(x,s)=1+\frac{s}{1-e^{-1}}$,\ $B(x,s)=2e^x$, \ $x \in (0,1),\ s\geq 0$, $g(t)=\frac{e^{-1}}{1+e^{-t}}$, $t>0$.}\medskip \\
We can observe  that $d$ and $B$ satisfy hypotheses of Theorem \ref{convergence}. Therefore  \eqref{new_scheme} is a convergent  numerical scheme. On the other hand, it is easy to check that  for the given set of functions, the function $u(x,t)=\frac{e^{-x}}{1+e^{-t}}$ is a solution to \eqref{non_e1}. 
\begin{figure}[h!]
	\centering
	\begin{subfigure}[b]{0.49\linewidth}
		\includegraphics[width=\linewidth]{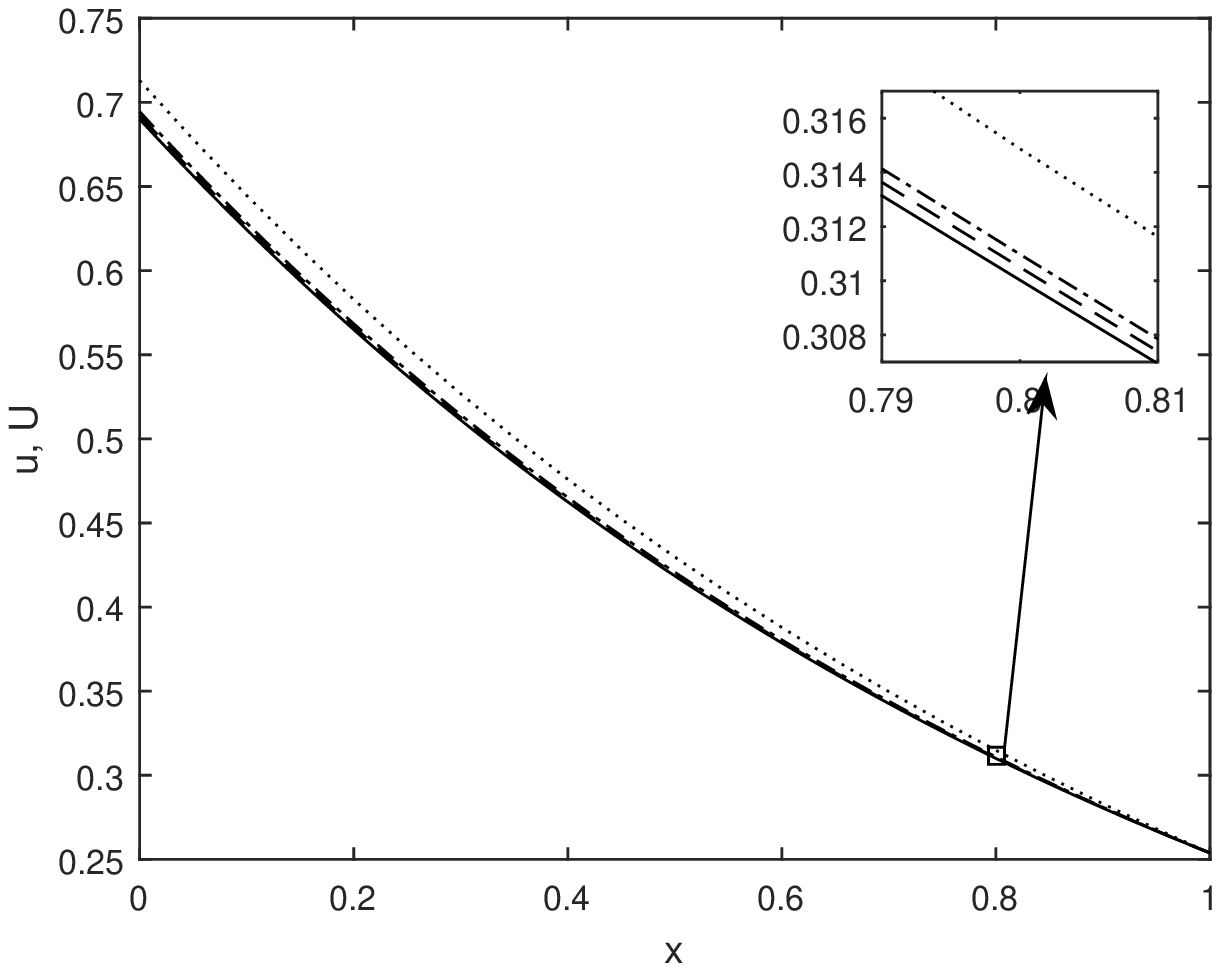}
		\label{fig:steady1}
	\end{subfigure}
	\begin{subfigure}[b]{0.49\linewidth}
		\includegraphics[width=\linewidth]{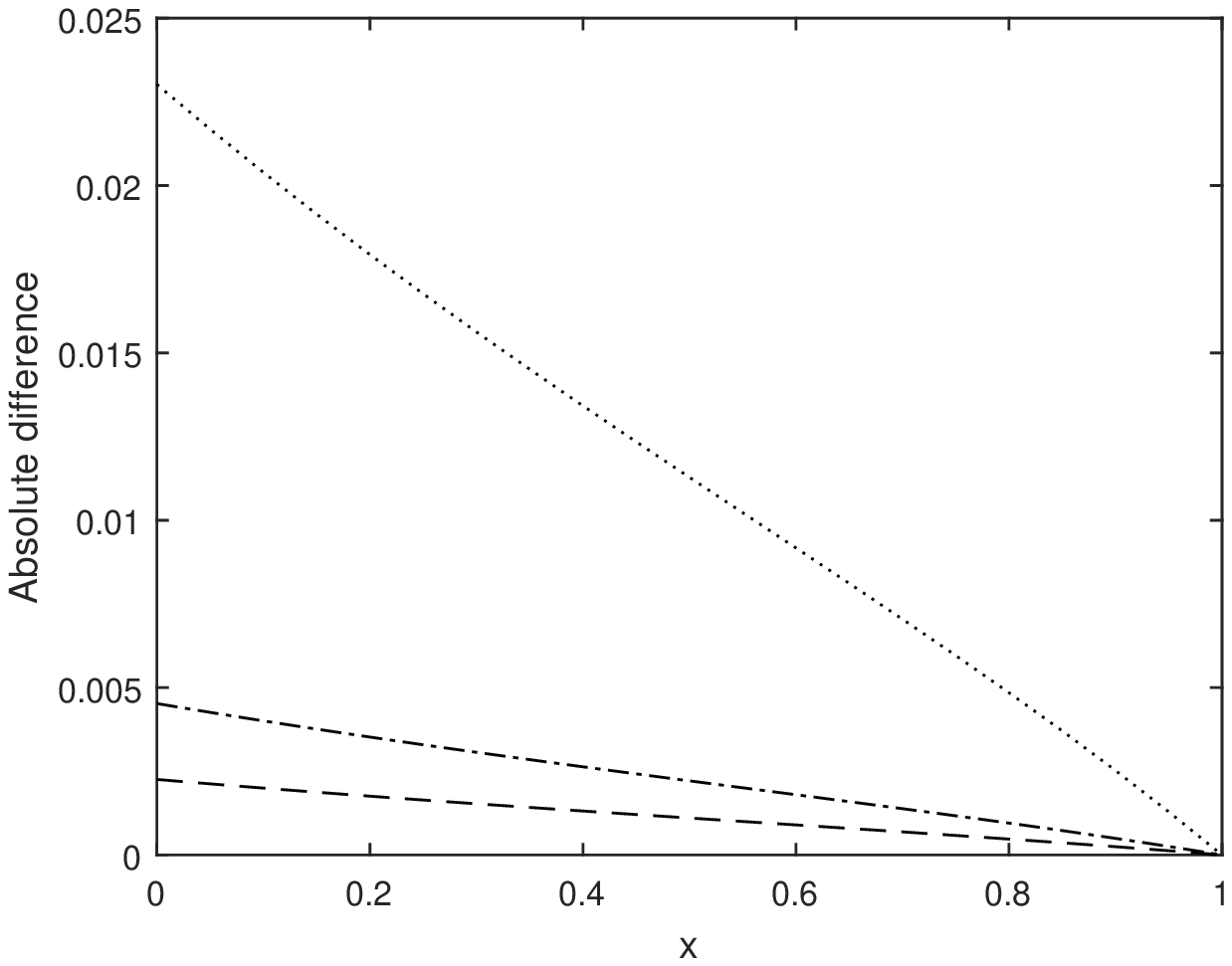}
		\label{fig:error1}
	\end{subfigure}
	\caption{The exact solution to \eqref{non_e1} and the approximate solutions to \eqref{new_scheme}  with $d(x,s)$, $B(x,s)$ given in Example \ref{example3}; Left: $u(x,0.2)$ (solid line), $\boldsymbol{U}_{0.05}$ (dotted line), $\boldsymbol{U}_{0.01}$ (dash-dotted line), $\boldsymbol{U}_{0.005}$ (dashed line) for $0\leq x \leq 1$, Right: $\mid u(x,0.2)- \boldsymbol{U}_{0.05}\mid$ (dotted line), $\mid u(x,0.2)- \boldsymbol{U}_{0.01}\mid$ (dash-dotted line) and $\mid u(x,0.2)- \boldsymbol{U}_{0.005}\mid$ (dashed line).}
	\label{1fig2}
\end{figure}
We display the exact solution $u$ to \eqref{non_e1} and the numerical solutions $U$ to  \eqref{new_scheme} in Figure \ref{1fig2}.
In Figure \ref{1fig2} (left), the exact solution $u$ to \eqref{non_e1} and numerical solutions to \eqref{new_scheme} with $h=0.05,\ 0.01, \ 0.005$ at $t=0.8$ are presented. From this figure, it is evident that $\boldsymbol{U}_{0.05}$, $\boldsymbol{U}_{0.01}$ and $\boldsymbol{U}_{0.005}$ are approaching to $u$  at $t=0.8$. 
  In Figure \ref{1fig2} (right), we show the absolute difference between $u$ and $\boldsymbol{U}_{h}$ at $t=0.8$, with $h=0.05,\ 0.01,\ 0.005$. We can see from these figures that the sequence of numerical solutions $\boldsymbol{U}_h$  indeed converges to the solution $u$ at $t=0.8$ as $h$ tends to 0.
\section*{Conclusions}
We have proposed a finite difference numerical scheme to the McKendrick-Von Foerster equation with diffusion \eqref{e1} in which  at the boundary point $x=0$ Robin condition is prescribed, and the Dirichlet condition is given at $x=a_\dagger$. Furthermore, we have proved that the proposed numerical scheme is stable restricted to the thresholds $R_h$. Moreover,  we have established that the given scheme is indeed convergent using a result due to Stetter. Using the similar  technique, one can easily obtain a convergent scheme when \eqref{e1} has nonlinear, nonlocal Neumann boundary condition at $x=0$. The result can be extended to M-V-D with nonlocal nonlinear Dirichlet boundary conditions at both the end points in a bounded domain. However, it is an interesting problem to design a convergent scheme for \eqref{e1} in the unbounded domain $[0,\infty)$.
\section*{Acknowledgements}
\noindent The first author would like to thank CSIR (award number: 09/414(1154)/2017-EMR-I) for providing the financial support for his research.
The second author is supported by Department of Science and Technology, India under MATRICS (MTR/2019/000848).
\bibliographystyle{abbrv}
\bibliography{references}
\end{document}